\documentclass[12pt]{article}
\usepackage{amsmath, amsfonts, latexsym, amssymb, graphicx, mathtools,  wasysym,}

\usepackage{hyperref}

\hypersetup{colorlinks,citecolor=blue,linkcolor=blue}

\title{ Semistability of Graph Products}
\author{Michael Mihalik}
\newtheorem{theorem}{Theorem}[section]

\newtheorem{lemma}[theorem]{Lemma}
\newtheorem{corollary}[theorem]{Corollary}
\newtheorem{remark}[theorem]{Remark}

\newcounter{claimnum}

\newcounter{definitionnum}
\newenvironment{definition}{\addvspace{12pt}\refstepcounter{definitionnum}
\noindent{\bf Definition \arabic{definitionnum}.}}{\par\addvspace{12pt}}

\newenvironment{proof}{\addvspace{12pt}\noindent{\bf Proof:}}{
$\Box$\par\addvspace{12pt}}

\newcounter{examplenum}

\date{\today}
\begin{document}

\maketitle

\begin{abstract}
A {\it graph product} $G$ on a graph $\Gamma$ is a group defined as follows: For each vertex $v$ of $\Gamma$ there is a corresponding  non-trivial group $G_v$. The group $G$ is the quotient of the free product of the $G_v$ by the commutation relations $[G_v,G_w]=1$ for all adjacent $v$ and $w$ in $\Gamma$. A finitely presented group $G$ has {\it semistable fundamental group at $\infty$} if for some (equivalently any) finite connected CW-complex $X$ with $\pi_1(X)=G$, the universal cover $\tilde X$ of $X$ has the property that any two proper rays in $\tilde X$ are properly homotopic. The class of finitely presented groups with semistable fundamental group at $\infty$  is known to contain many other classes of groups, but it is a 40 year old question as to whether or not all finitely presented groups have semistable fundamental group at $\infty$.  Our main theorem is a combination result. It states that if $G$ is a graph product on a finite graph $\Gamma$ and each vertex group is finitely presented, then $G$ has non-semistable fundamental group at $\infty$ if and only if there is a vertex $v$ of $\Gamma$ such that $G_v$ is not semistable, and the subgroup of $G$ generated by the vertex groups of vertices adjacent to $v$ is finite (equivalently $lk(v)$ is a complete graph and each vertex group of $lk(v)$ is finite). Hence if one knows which vertex groups of $G$ are not semistable and which are finite, then an elementary inspection of $\Gamma$ determines whether or not $G$ has semistable fundamental group at $\infty$. 
\end{abstract}

\section{Introduction} \label{intro}
Given a graph $\Gamma$ with vertex set $V(\Gamma)$, and a group $G_v$ for each $v\in V(\Gamma)$, the graph product for $(\Gamma, \{G_v\}_{v\in V(\Gamma)})$
is the quotient of the free product of the $G_v$ by the normal closure of the set of all commutators $[g,h]$ where $g$ and $h$ are elements of adjacent vertex groups. Every right angled Coxeter group and right Angled Artin group is a graph product where vertex groups are copies of $\mathbb Z_2$ and $\mathbb Z$ respectively. As an immediate corollary to our theorem we have that all right angled Artin and Coxeter groups have semistable fundamental group at $\infty$ (a result first proved in \cite{M96}).

The question of whether or not all finitely presented groups have semistable fundamental group at $\infty$ has been studied for over 40 years and is one of the premier questions in the asymptotic theory of finitely presented groups. If a finitely presented group $G$ has semistable fundamental group at $\infty$ then $H^2(G;\mathbb ZG)$ is free abelian. The question of whether or not all finitely presented groups $G$ are such that $H^2(G;\mathbb ZG)$ is free abelian is attributed to H. Hopf and remains unanswered. Semistability is a quasi-isometry invariant of a group \cite{BR93}  or \cite{G} and the class of groups with semistable fundamental group at $\infty$ contains many classes of groups, including: Word hyperbolic groups (combining work of Bestvina-Mess \cite{BM91}, Bowditch \cite{Bow99B} and \cite{Swarup}), 1-relator groups \cite{MT92}, general Coxeter and Artin groups, most solvable groups \cite{M4}  and various group extensions \cite{M6} and ascending HNN extensions  \cite{M7}.

There are two important combination results, both of which are used in the  proof of our main theorem: 

\begin{theorem}\label{MT}(M. Mihalik, S. Tschantz \cite{MT1992}) %{\bf (MT)} 
If $G$ is the amalgamated product $A\ast_CB$ where $A$ and $B$ are finitely presented with semistable fundamental group at $\infty$, and $C$ is finitely generated, then $G$ has semistable fundamental group at $\infty$. 
\end{theorem}
The next result implies the previous one if $A$ and $B$ are 1-ended and $C$ is infinite.
  
\begin{theorem}\label{MM} (M. Mihalik \cite{M4}) %{\bf (MM)} 
Suppose  $A$ and $B$ are 1-ended finitely generated subgroups of the finitely generated group $G$ such that $A$ and $B$ have semistable fundamental group at $\infty$, the set $A\cup B$ generates $G$ and the group $A\cap B$ contains a finitely generated infinite subgroup. Then $G$ has semistable fundamental group at $\infty$. 
\end{theorem}

Our main theorem follows:
\begin{theorem}\label{main} %{\bf (main)}
Suppose $G$ is a graph product on the finite connected graph $\Gamma$ where each vertex group is finitely presented. Then $G$ does not have semistable fundamental group at $\infty$ if and only if there is a vertex $v$ of $\Gamma$ such that: 

(1) $G_v$ does not have semistable fundamental group at $\infty$ and  

(2) the link of $v$ is a complete graph with each vertex group finite.
\end{theorem} 

If $\Gamma_1$ is a subgraph of $\Gamma$ then let $\langle \Gamma_1\rangle$ denote the subgroup of $G$ generated by the vertex groups of $\Gamma_1$. For $v$ a vertex of $\Gamma$, let $lk(v)$ (the link of $v$)  be the full subgraph of $\Gamma$ on the vertices adjacent to $v$. Let $st(v)$ (the star of $v$) be $\{v\}\cup lk(v)$. 
Results in Section \ref{GP} show that when the graph product $G$ is not semistable (and $\Gamma$ is not complete), then $G$ ``visually" splits as an amalgamated product $G=\langle st(v)\rangle \ast_{\langle lk(v)\rangle} \langle \Gamma-\{v\}\rangle$, where $v$ is a vertex of $\Gamma$, $G_v$ is not semistable and $\langle lk(v)\rangle$ is a finite group. 
The theory of ends of finitely generated graph products is completely worked out by O. Varghese.

\begin{theorem} \label{OV} (O. Varghese, \cite{OV})  %{\bf (OV)}
Suppose $G$ is a finitely generated graph product group on the graph $\Gamma$ (so that each $G_v$ is finitely generated and $\Gamma$ is finite) and $G$ has more than one end, then either:

(i) $\Gamma$ is a complete graph such that one vertex group has more than one end and all others are finite, or

(ii) $G$ visually splits over a finite group. This means that there are non-empty full subgraphs $\Gamma_1$  and $\Gamma_2$ of $\Gamma$  (neither containing the other) such that $\Gamma=\Gamma_1\cup \Gamma_2$, (so $\Gamma_1\cap \Gamma_2$ separates $\Gamma$) and $\langle \Gamma_1\cap \Gamma_2\rangle$ is a finite group. In particular, $G$ visually decomposes as the (non-trivial) amalgamated product
$$G \cong \langle \Gamma_1\rangle\ast _{\langle\Gamma_1\cap \Gamma_2\rangle} \langle\Gamma_2\rangle$$  
\end{theorem}

Condition (ii) is equivalent to the condition that there is complete separating subgraph $\Delta$ of $\Gamma$ such that each vertex group of $\Delta$ is finite.   
For graph products of groups, Theorem \ref{OV} is a visual version of J. Stallings splitting theorem for general finitely generated groups. Stallings theorem states that a finitely generated group has more than 1-end if and only if it splits non-trivially over a finite group.

The following corollary is a combination result. Part (1) follows directly from Theorem \ref{main} and Part (2) follows directly from Theorems \ref{main} and \ref{OV}.
 
\begin{corollary} \label{C1}%{\bf (C1)} 
Suppose $G$ is a graph product on the finite connected graph $\Gamma$ and each vertex group of $\Gamma$ is finitely presented. If either of the following conditions are met, then $G$ has semistable fundamental group at $\infty$:

 (1) Each vertex group has semistable fundamental group at $\infty$.
 
  (2) The group $G$ is 1-ended and $\Gamma$ is not complete. 
\end{corollary}

The remainder of the paper is organized as follows: Section \ref{GP} contains basic results on graph products of groups. These results are all based on ``standard presentations" of graph products of groups.  Section \ref{SS} contains definitions and  results on groups with semistable fundamental group at $\infty$. If $G$ is a finitely presented 1-ended group with semistable fundamental group at $\infty$ we define the fundamental group at $\infty$ for $G$. The proof of the main theorem is produced in Section \ref{PF}.  Finally a discussion of possible generalizations of the main theorem from finitely presented to finitely generated groups is discussed in Section \ref{FG}. There are finitely generated groups that are known to have non-semistable fundamental group at $\infty$. 

\section{Graph Products of Groups} \label{GP} 
Let $V(\Gamma)$ be the vertices and $E(\Gamma)$ the edges of a graph $\Gamma$ (so $E\subset V\times V$). Suppose that for each vertex $v\in V(\Gamma)$, $G_v$ is a group.  The {\it graph product} for $(\Gamma, \{G_v\}_{v\in V(\Gamma)})$  is the quotient of the free product of the $G_v$ for $v\in V(\Gamma)$ by the normal closure of the set of all commutators $[a,b]$ where $a\in G_v$, $b\in G_w$ and $(v,w)\in E(\Gamma)$.  For $v\in \Gamma$, let $\mathcal P_v=\langle S_v:R_v\rangle$ be a presentation for $G_v$. The {\it standard presentation} for $G$ corresponding to  $\{\mathcal P_v\}_{v\in V(\Gamma)}$ is 
$$\langle S:R\rangle$$ 
where  $R=(\cup_{v\in V(\Gamma)} R_v)\cup \{[a,b]:a\in S_v, b\in S_w \hbox{ where } (v,w)\in E(\Gamma)\}$ and $S=\cup_{v\in V(\Gamma)} S_v$. 

If $\Gamma_1$ is a full subgraph of $\Gamma$ we denote the subgroup of $G$ generated by the $G_v$ for $v\in \Gamma_1$ by $\langle \Gamma_1\rangle_G$ (or simply $\langle \Gamma_1\rangle$ when the over group $G$ is evident).  If $G_1$ is the graph product for $(\Gamma_1,\{G_v\}_{v\in V(\Gamma_1)})$  there is a natural homomorphism $m:G_1\to G$ with image $\langle \Gamma_1\rangle_G$ (induced by the inclusion of $\Gamma_1$ in $\Gamma$).  We need to recall a few basic facts about graph products. 

\begin{lemma} \label{sub}%{\bf(sub)}
Suppose $G$ is the graph product for $(\Gamma, \{G_v\}_{v\in V(\Gamma)})$.
If $\Gamma_1$ is a full subgraph of $\Gamma$, then the subgroup $\langle \Gamma_1\rangle$ of $G$ is isomorphic to  $G_1$, the graph product for $(\Gamma_1,\{G_v\}_{v\in \Gamma_1})$. 

In fact, if $m:G_1\to G$ is the natural homomorphism with image $\langle\Gamma_1\rangle$ and $q$ is the quotient homomorphism of $G$ by the normal closure of the union of the $G_v$ where $v\not\in V(\Gamma_1)$ 
$$G_1{\buildrel m \over \longrightarrow} \ G\ {\buildrel q\over \twoheadrightarrow}\  G/\langle \langle\cup_{v\not\in V(\Gamma_1)}G_v\rangle \rangle$$
then the image of $q$ is isomorphic to $G_1$ and the composition $qm$ is an isomorphism. 
In particular, $\langle \Gamma_1\rangle (\cong G_1)$ is a retract of $G$.    
\end{lemma}
\begin{proof} For each $v\in V(\Gamma)$ let $\mathcal P_v=\langle S_v:R_v\rangle$ be a presentation for $G_v$. Let $\mathcal P=\langle S:R\rangle$ be the presentation for $G$ corresponding to the $\mathcal P_v$. The quotient group of $G$ by the normal closure of the $G_v$ with $v\not\in V(\Gamma_1)$ has presentation $\langle S: R\cup (\cup_{v\not\in V(\Gamma_1)}S_v )\rangle$. Tietze moves that eliminate the generators of all $S_v$ for $v\not\in V(\Gamma_1)$ leave a presentation for $\langle \Gamma_1\rangle$. This presentation is the presentation for $G_1$ corresponding to the $\mathcal P_v$ for $v\in V(\Gamma_1)$ and $mq$ is the identity on the generators of this presentation for $G_1$ and hence an isomorphism. 
\end{proof}

Lemma \ref{sub} applied to a standard presentation of a graph product implies: 
 
\begin{lemma}\label{FGP}%{\bf (FGP)} 
If $G$ is the graph product on $(\Gamma, \{G_v\}_{v\in V(\Gamma)})$ then $G$ is finitely generated (presented) if and only if $\Gamma$ is a finite graph and $\Gamma_v$ is finitely generated (presented) for each $v\in V(\Gamma)$.
\end{lemma} 

\begin{lemma}\label{split} %{\bf (split)} 
Suppose $\Delta$ is a subgraph of $\Gamma$ that separates the vertices of $\Gamma-\Delta$ into two non-empty disjoint sets $A$ and $B$ and no vertex of $A$ is connected to a vertex of $B$ by an edge. Let $\Gamma_A$ be the full subgraph of $\Gamma$ on the set of vertices in $A\cup \Delta$ and similarly for $\Gamma_B$. Let $G_A$, $G_B$ and $G_\Delta$ be the graph products on $\Gamma_A$, $\Gamma_B$ and $\Delta$ respectively, with vertex groups from $\Gamma$. Then $G$ ``visually" decomposes as $G_A\ast_{G_\Delta} G_B$. 
\end{lemma} 

\begin{proof}
For each vertex $v\in \Gamma$, let $\mathcal P_v$ be the presentation for $G_v$ with generators $S_v$ and relations $R_v$. A  presentation for $G_A$ is $\langle S_A:R_A\rangle$ where $R_A=$
$$(\cup_{v\in (A\cup \Delta)} R_v) \cup \{ [a,b]:a\in S_v, b\in S_w \hbox{ where } v,w\in A\cup \Delta \hbox{ and } (v,s)\in E(\Gamma)\}$$ 
$$S_A=\cup_{v\in (A\cup \Delta)} S_v$$ 
Similarly one obtains a presentation for $G_B$ as $\langle S_B:R_B\rangle$ and for $G_\Delta$ as $\langle S_\Delta:R_\Delta\rangle$. The important thing to observe is there is no edge in $\Gamma$ from a vertex of $A$ to a vertex of $B$ and if $E$ is an edge of $\Gamma$ it is an edge of the full subgraph of $\Gamma$ on $A\cup \Delta$ or $B\cup \Delta$, or both - in which case $E$ is an edge of $\Delta$. Note that $S_A\cap S_B=S_\Delta$.  Hence $\langle S_A \cup S_B:R_A\cup R_B\rangle$ is the usual presentation for the amalgamated product $G_A\ast_{G_\Delta}G_B$, obtained from the presentations $\langle S_A:R_A\rangle$ and $\langle S_B:R_B\rangle$ by identifying generators of $G_\Delta$ in $S_A$ with their counterparts in $S_B$. This last presentation is also the presentation for $G$ corresponding to the presentations $\mathcal P_v$ for $v\in V(\Gamma)$.  
\end{proof}

\section{Semistability Preliminaries}\label{SS}

 A continuous map between topological spaces $m:X\to Y$ is {\it proper} if for each compact set $C\subset Y$, $m^{-1}(C)$ is compact in $X$. Given proper maps $f,g:X\to Y$, we say $f$ is properly homotopic to $g$ if there is a proper map $H:X\times [0,1]\to Y$ such that $H(x,0)=f(x)$ and $H(x,1)=h(x)$ for all $x\in X$. 

The notion of semistable fundamental group at $\infty$ makes sense for a wide range of topological spaces. We are  only interested in locally finite connected $CW$ complexes in this article. R. Geoghegan's book \cite{G} is an excellent source for the basic theory of semistable spaces and groups. Suppose $X$ is a locally finite connected CW complex. Two proper rays $r,s:[0,\infty)\to X$ {\it converge to the same end} of $X$ if for any compact $C\subset X$,  $r$ and $s$ eventually stay in the same component of $X-C$. We say $X$ has {\it semistable fundamental group at $\infty$} if any two proper rays $r,s:[0,\infty)\to X$ that converge to the same end of $X$ are properly homotopic. Suppose $X$ has 1-end.  For each $i\in \{0,1,\ldots\}$, suppose $C_i$ and $D_i$ are compact sets such that $C_i$ (respctively $D_i$) is a subset of the interior of $C_{i+1}$ (respectively $D_{i+1}$). Say $r,s:[0,\infty)\to X$ are proper rays. If $X$ has semistable fundamental group at $\infty$ then the inverse systems $\pi_1(X-C_i,r)$ and $\pi_1(X-D_i,s)$ (with bonding maps induced by inclusion) are pro-isomorphic %(see the discussion following Proposition 16.2.3, \cite{G}) 
and the inverse limit of such a system is the {\it fundamental group at $\infty$} of $X$.   

The finitely presented group $G$ has semistable fundamental group at $\infty$ if for some (equivalently any) finite connected CW complex $X$, with $\pi_1(X)=G$, the universal cover of $X$ has semistable fundamental group at $\infty$. For simplicity we say $G$ is semistable. If $G$ is 1-ended and semistable, then the fundamental group at $\infty$ of $G$ is the fundamental group at $\infty$ of the universal cover of $X$. Again, this is independent of finite complex $X$ as long as $\pi_1(X)=G$.

Over the last 40 years a substantial theory has been built to study the semistability of finitely presented groups.  It is as yet unknown if all finitely presented groups are semistable. Using M. Dunwoody's accessibility theorem for finitely presented groups, the problem is reduced to the same problem for 1-ended groups \cite{M87}. For our purposes, we only need a few of the results in the literature. In \cite{BR93} S. Brick showed that semistablity is a quasi-isometry invariant of finitely presented groups. This immediately implies:
 
 \begin{theorem} \label{FI} %{\bf (FI)} 
 If $A$ is a finitely presented group and $A$ has finite index in the group $B$, then $A$ has semistable fundamental group at $\infty$ if and only if $B$ has semistable fundamental group at $\infty$.
 \end{theorem} 
 
 \begin{theorem}\label{prod}  (Theorem 2.2, \cite{M1}) %{\bf (prod)}
 Suppose $X$ and $Y$ are locally finite connected and infinite CW-complexes. Then $X\times Y$ is 1-ended with semistable fundamental group at $\infty$. In particular, if $G$ and $H$ are finitely presented infinite groups, then $G\times H$ is 1-ended with semistable fundamental group at $\infty$.
 \end{theorem}
 
The next result is proved in \S\ref{SimpApp}
 \begin{theorem} \label{splitF} %{\bf (splitF)} 
 Suppose the group $G$ splits as $A\ast_C B$, where $C$ is finite and $A$ and $B$ are finitely presented. If $A$ does not have semistable fundamental group at  $\infty$, then  $G$ does not have semistable fundamental group at $\infty$. 
 \end{theorem} 

\section{Proof of the Main Theorem}\label{PF}
\noindent{\bf Main Theorem} {\it Suppose $G$ is a graph product on the finite connected graph $\Gamma$ where each vertex group is finitely presented. Then $G$ does not have semistable fundamental group at $\infty$ if and only if there is a vertex $v$ of $\Gamma$ such that: 

(1) $G_v$ does not have semistable fundamental group at $\infty$ and  

(2) the link of $v$ is a complete graph with each vertex group finite.}

\begin{proof} If there is a vertex $v$ of $\Gamma$ such that conditions (1) and (2) of the theorem hold for $v$, then $\langle lk(v)\rangle$ is a finite group. If additionally, $\Gamma=st(v)$ then $G=G_v\times \langle lk(v)\rangle$ and $G$  does not have semistable fundamental group at $\infty$ by Theorem \ref{FI}. If $\Gamma\ne st(v)$, then by Lemma \ref{split}, $G$ visually decomposes as $\langle \Gamma-\{v\}\rangle\ast_{\langle lk(v)\rangle} \langle st(v)\rangle$ where $\langle st(v)\rangle=G_v\times \langle lk(v)\rangle$. Combining this with Theorem \ref{splitF}, we immediately see that $G$ does not have semistable fundamental group at $\infty$. The ``if" portion of the Theorem is proved.

Assume that $G$ is a graph product on the graph $\Gamma$ and no vertex $v$ of $\Gamma$ is such that conditions (1) and (2) of the theorem hold for $v$. We must show that $G$ has semistable fundamental group at $\infty$. If there is such a non-semistable graph product $G$, on a graph $\Gamma$ such that no vertex of $\Gamma$ satisfies conditions (1) and (2) of the Theorem, then assume that $\Gamma$ has as few vertices as possible.  

\begin{lemma} \label{basic}%{\bf (basic)}
The graph $\Gamma$ and the group $G$ satisfy the following:

(A) The graph $\Gamma$ is not equal to $st(w)$ for any vertex $w$. 

(B) The group $G$ is 1-ended. In particular for each vertex $w$ of $\Gamma$, $\langle lk(w)\rangle$ is infinite. 

(C) There is at least one vertex $w$ in $\Gamma$ such that $G_w$ is not semistable. 

(D) If $v$ is a vertex of $\Gamma$ and $G_v$ is infinite, then $\langle st(v)\rangle$ is 1-ended and semistable.
\end{lemma}
\begin{proof} (Of Part (A))
Assume $\Gamma=st(w)$.  Then $G=G_w\times \langle lk(w)\rangle$. Since $G$ is not semistable, Theorem \ref{prod} implies either $G_w$ or  $\langle lk(w)\rangle$ is finite (and so the other has finite index in $G$).  If $\langle lk(w)\rangle$ is finite, then by Theorem \ref{FI}, $G_w$ is not semistable. But then $w$ satisfies conditions (1) and (2) of the Theorem - contrary to hypothesis. 

 If $G_w$ is finite, then $\langle lk(w)\rangle$ has finite index in $G$ and so $\langle lk(w)\rangle$ is not semistable. Let $lk_1(v)$ denote the link of any vertex $v$ in $lk(w)$. By the minimality of $\Gamma$, there is a vertex $v\in lk(w)$ such that $G_v$ is not semistable and $\langle lk_1(v)\rangle$ is finite. But, then each vertex group of the complete graph $lk(v)=lk_1(v)\cup \{w\}$ is finite and $v$ satisfies conditions (1) and (2) - contrary to hypothesis. %Since no vertex of $\Gamma$ satisfies conditions (1) and (2) of the Theorem, (and since $G_w$ is finite), no vertex of $lk(w)$ satisfies conditions (1) and (2) of the Theorem. By the minimality of $\Lambda$, it must be that $\langle lk(w)\rangle$ is semistable (which it is not). 
 Instead, part (A) is verified.   
 
 (Part (B)) Suppose $G$ has more than 1-end.  Part (A) implies $\Gamma$ is not complete. Theorem \ref{OV} implies there are non-empty full subgraphs $\Gamma_1$ and $\Gamma_2$ of $\Gamma$ such that $\Gamma_1\cup \Gamma_2=\Gamma$, $\Gamma_1\cap \Gamma_2$ separates $\Gamma$ and $\langle \Gamma_1\cap\Gamma_2\rangle$ is finite. Since $G=\langle \Gamma_1\rangle\ast_{\langle \Gamma_1\cap\Gamma_2\rangle} \langle\Gamma_2\rangle$, Theorem \ref{MT} implies either $\langle \Gamma_1\rangle$ or $\langle \Gamma_2\rangle$ is not semistable. Assume $\langle \Gamma_1\rangle$ is not semistable. Let $lk_1$ denote link in $\Gamma_1$. By the minimality of $\Gamma$ there is a vertex $v$ of $\Gamma_1$ such that $G_v$ is not semistable and $\langle lk_1(v)\rangle$ is finite. Since each vertex group of $\Gamma_1\cap \Gamma_2$ is finite, $v\not\in \Gamma_1\cap\Gamma_2$. Since there is no edge connecting a vertex of $\Gamma_1-(\Gamma_1\cap\Gamma_2)$ to a vertex of $\Gamma_2-(\Gamma_1\cap\Gamma_2)$, $lk_1(v)=lk(v)$. Then $v$ satisfies conditions (1) and (2) in $\Gamma$ - contrary to hypothesis.   Part (B) is verified. 

(Part (C)) Next assume $G_w$ is semistable for every vertex $w$ of $\Gamma$. If $w$ a vertex of $\Gamma$, then by Part (A), $\langle \Gamma-\{w\}\rangle\ne lk(w)$ and by Lemma \ref{split}, $G=\langle \Gamma-\{w\}\rangle \ast_{\langle lk(w)\rangle} \langle st(w)\rangle$. Since every vertex group of $\Gamma$ is semistable, our minimality condition on $\Gamma$ implies both $\langle \Gamma-\{v\}\rangle$ and $\langle st(v)\rangle$ are semistable. By Theorem \ref{MT}, $G$ is semistable, contrary to assumption. Instead, Part (C) is verified.

(Part (D)) Assume $G_v$ is infinite. By Part (B), $\langle lk(v)\rangle$ is infinite. By Theorem \ref{prod}, the group $\langle st(v)\rangle=G_v\times \langle lk(v)\rangle$ is 1-ended and  semistable. 
\end{proof}

\begin{lemma} \label{adj} %{\bf (adj)}
If $v$ is a vertex of $\Gamma$ then 
$v$ is adjacent to a vertex $w$ such that $G_w$ is not semistable. 
\end{lemma}
\begin{proof}
Suppose $v$ is a vertex of $\Gamma$ and every vertex $w\in lk(v)$ is such that $G_w$ is semistable. Let $\Gamma_1$ be the full subgraph of $\Gamma$ on the vertices of $\Gamma-\{v\}$. If $u$ is a vertex of $\Gamma_1$ such that $G_u$ is not semistable, then $u\in\Gamma_1-lk(v)$, and so the link of $u$ in $\Gamma_1$ agrees with the link of $u$ in $\Gamma$. By the minimality of $\Gamma$, the group $\langle \Gamma_1\rangle$ is semistable. 

If $G_v$ is semistable, and $u$ is a vertex of $st(v)$ then $G_u$ is semistable.  The minimality of $\Gamma$ and Lemma \ref{basic}(A) imply that $\langle st(v)\rangle$ is semistable.  If $G_v$ is not semistable, then Lemma \ref{basic}(D) implies that $\langle st(v)\rangle $ is semistable. We have that in any case, both $\langle \Gamma_1\rangle$ and $\langle st(v)\rangle$ are semistable.  As $G=\langle \Gamma_1\rangle\ast_{\langle lk(v)\rangle} \langle st(v)\rangle$, $G$ is semistable by Theorem \ref{MT} - contrary to hypothesis. 
\end{proof}

\begin{lemma} \label{star} %{\bf (star)}
Suppose $v$ is a vertex of $\Gamma$ such that $G_v$ is not semistable. Let $\Gamma_1$ be the full subgraph of $\Gamma$ on the vertices of $\Gamma-\{v\}$. The group $G_{\Gamma_1}$ is not semistable.
\end{lemma} 
\begin{proof}
Suppose $G_{\Gamma_1}$ is semistable. Note that $G=G_{\Gamma_1}\ast_{\langle lk(v)\rangle} \langle st(v)\rangle$. Since $st(v)$ is semistable (Lemma \ref{basic} (D)), Theorem \ref{MT} implies $G$ is semistable, the desired contradiction.
\end{proof}

Let $v$ be a vertex of $\Gamma$ such that $G_v$ is not semistable. Let $\Gamma_1$ be the full subgraph of $\Gamma$ on the vertices of $\Gamma-\{v\}$ and $G_1=\langle\Gamma_1\rangle$. For $w$ a vertex of $\Gamma_1$, let $lk_1(w)$ be the link of $w$ in $\Gamma_1$. 

$(\ast)$ If $w$ is a vertex of $\Gamma_1$, then $lk_1(w)=lk(w)$ if $w\not\in lk(v)$. 

By Lemma \ref{star}, $G_1$ is not semistable. Since $\Lambda_1$ has fewer vertices than $\Gamma$, there is a vertex $w$ of $\Gamma_1$ such that $G_w$ is not semistable and $\langle lk_1(w)\rangle$ is finite. By $(\ast)$, $w\in lk(v)$ and so  
$$lk(w)=\{v\}\cup lk_1(w)$$ 

Let $\Gamma_2$ be the full subgraph of $\Gamma$ on the vertices of $\Gamma-\{w\}$ and $G_2=\langle \Gamma_2\rangle$. For $u$ a vertex of $\Gamma_2$, let $lk_2(u)$ be the link of $u$ in $\Gamma_2$. As argued above (for $v$) there is a vertex $z\in lk(w)$ such that $G_z$ is not semistable and $\langle lk_2(z)\rangle$ is finite. The only vertex of $lk(w)(=\{v\}\cup lk_1(v))$ with non-semistable group is $v$. So $\langle lk_2(v)\rangle$ is finite and 
$$lk(v)=\{w\}\cup lk_2(v)$$
By Lemma \ref{basic}(D), the groups $\langle st(v)\rangle $ and $\langle st(w)\rangle $ are 1-ended and semistable. Since $v,w\in st(v)\cap st(w)$, Theorem \ref{MM} implies the group  $S=\langle st(v)\cup st(w)\rangle$ is 1-ended and semistable. In particular, $G\ne S$. Let $\Gamma_3$ be the full subgraph of $\Gamma$ on the vertices of $\Gamma-\{v,w\}$. For $z$ a vertex of $\Gamma_3$, let $lk_3(z)$ be the link of $z$ in $\Gamma_3$. Let $G_3=\langle \Gamma_3\rangle$. If $z$ is a vertex of $\Gamma_3$ such that $G_z$ is not semistable, then $z\not\in lk(w)\cup lk(v)$ (this set only contains vertices whose groups are finite). That means $lk(z)=lk_3(z)$. By the minimality of $\Gamma$, $G_3$ is semistable. Since $(lk(w)\cup lk(v))-\{v,w\}$ separates $w$ and $v$ from the vertices of $\Gamma-(st(v)\cup st(w))$, we have 
$$G=G_3\ast_{\langle (lk(w)\cup lk(v))-\{v,w\}\rangle}(st(v)\cup st(w))$$ 
By Theorem \ref{MT}, $G$ is semistable - the desired contradiction.
\end{proof}

\section{Generalizations to Finitely Generated Groups}\label{FG}

In \cite{M4}  we extended the definition of semistability at $\infty$ to finitely generated groups. Nearly all of the major  semistability results for finitely presented groups generalize to finitely generated groups. 

If a group $G$ has finite presentation $\mathcal P=\langle S:R\rangle$ the standard 2-complex $X(\mathcal P)$ for $\mathcal P$ has a single vertex $v$ and one loop at $v$ (labeled $s$) for each $s\in S$. If $r\in R$ then a 2-cell is attached to this wedge of loops at $v$ according to the labeling of $r$. The resulting 2-complex is $X(\mathcal P)$ and $\pi_1(X)=G$. The universal cover $\tilde X$ of $X$ is the {\it Cayley 2-complex} for $\mathcal P$. The 1-skeleton of $\tilde X$ is the Cayley graph $\Lambda(S,G)$.  At each vertex $g\in \Gamma$ there is a 2-cell with boundary equal to an edge path labeled by $r$ for each $r\in R$. The group $G$ has semistable fundamental group at $\infty$ if and only if $\tilde X(\mathcal P)$ has semistable fundamental group at $\infty$. This is independent of choice of finite presentation $\mathcal P$. 
 
If $G$ is a finitely generated group with finite generating set $S$, let 
$\Lambda(S,G)$ be the Cayley graph for $(S,G)$.  Suppose $R$ is a finite collection of relations for $(S,G)$ (so $S$ is a subset of the free group on $S$). The 2-complex $\Lambda(S, G, R)$ has 1-skeleton $\Lambda(S,G)$. For each vertex $w\in \Lambda(G,S)$ and each $r\in R$ there is a 2-cell at $w$ with edge path boundary labeled by $r$. If $\mathcal P=\langle S:R\rangle$ were a finite presentation for $G$, then $\Gamma(S,G, R)$ would be the Cayley 2-complex for $\mathcal P$. If there is a finite collection of $S$-relations $R$ such that $\Gamma(S,G, R)$ has semistable fundamental group at $\infty$, then we say $G$ has semistable fundamental group at $\infty$. Again, this notion is independent of finite generating set for $G$ - If $S_1$ and $S_2$ are finite generating sets for $G$ then there is a finite set of $S_1$-relations $R_1$  for $G$ such that $\Gamma(S_1,G, R_1)$ has semistable fundamental group at $\infty$ if and only if there is a finite set of $S_2$-relations $R_2$ for $G$ such that $\Gamma(S_2,G,R_2)$ has semistable fundamental group at $\infty$ (see \cite{M4})
The main advantage of considering finitely generated groups with semistable fundamental group at $\infty$ is that they can be used to show certain finitely presented groups have semistable fundamental group at $\infty$ (see for instance \cite{M4} where solvable groups are shown to have semistable fundamental group at $\infty$.)

In order to prove our Main Theorem in the setting of finitely generated groups we would simply need the corresponding results for Theorems \ref{MT}, \ref{MM}, \ref{FI}, \ref{prod} and \ref{splitF} where the finitely presented hypothesis is relaxed to finitely generated.  With the exception of Theorem \ref{MT}, it is an easy exercise to check that the proofs of the other results extend to the finitely generated setting. The proof of Theorem \ref{MT} is a long technical argument and not so simple to analyze. 

\section{Proper Relative Simplicial Approximation}\label{SimpApp} %{\bf (SimpApp)}
In this section we review two results from \cite{MS19} that will help prove Theorem \ref{splitF}. All spaces are simplicial complexes and subcomplexes are full subcomplexes of the over complex. If $X$ is a simplicial complex, then we say a subcomplex $Z$ {\it separates} a vertex $v\in X-Z$ from a subcomplex $Y$ of $X$ if any edge path in $X$ from $v$ to a vertex of $Y$ contains a vertex of $Z$. 
Our model for the following theorem is when the space $X$ is the Cayley 2-complex for a group split non-trivially as $G=A\ast_CB$ where $A$ and $B$ are finitely presented and $C$ is finitely generated. 

We are interested in proper homotopies $M:[0,\infty)\times [0,1]\to  X$ of proper edge path rays  $r$ and $s$ into a connected locally finite simplicial 2-complex $X$, where $r$ and $s$ have image in a subcomplex $Y$ of $X$. Simplicial approximation allows us to assume that $M$ is simplicial (see Lemma \ref{simpA}). If $G$ is the amalgamated product $A\ast_CB$, then the space $X$ is the Cayley 2-complex for $G$, and the space $Y$ is the Cayley subcomplex of $X$ for $A$.  Say $\{Z_i\}_{i=1}^\infty$ is a collect of connected subcomplexes of $Y$ such that only finitely many $Z_i$ intersect any compact subset of $X$, and such that each $Z_i$ separate vertices of $Y$ from vertices of $X-Y$. In the $A\ast_CB$ setting, the $Z_i$ are the $aC$ cosets for each $a\in A$. 

The next result describes how to simplicially excise the parts of $[0,\infty)\times [0,1]$ that are not mapped into $Y$ (and perhaps a bit more is removed). What is removed from  $[0,\infty)\times [0,1]$ is a disjoint union of open sets $E_j$ for $j\in J$, each homeomorphic to $\mathbb R^2$ and $M([0,\infty)\times [0,1]-(\cup_{j\in J} E_j))\subset Y$. Also, $M$ maps the topological boundary of each $E_j$ into some $Z_i$.  When $E_j$ is bounded in $[0,\infty)\times [0,1]$ (contained in a compact set),  there is an embedded edge path loop $\alpha_j$ in $[0,\infty)\times [0,1]$ that bounds $E_j$, and $M(\alpha_j)$ has image in one of the $Z_i$. In this case the $E_j$ and $\alpha_j$ form a finite subcomplex of $[0,\infty)\times [0,1]$ homeomorphic a closed ball which contains a certain equivalence class of triangles that are mapped into $X-Y$. When $E_j$ is unbounded, $\alpha_j$ is an embedded proper edge path line that bounds $E_j$. Again $M(\alpha_j)$ has image in one of the $Z_i$ and $E_j$ and $\alpha_j$ form a subcomplex of $[0,\infty)\times [0,1]$ homeomorphic to the closed upper half plane.  Again, in this case $E_j$ will contain a certain equivalence class of triangles, each of which is mapped into $X-Y$. 
 
\begin{definition}
We call the pair $(E,\beta)$ a {\it disk pair} in the simplicial complex $[0,\infty)\times [0,1]$ if $E$ is an open subset of $[0,\infty)\times [0,1]$ homeomorphic to $\mathbb R^2$, $E$ is a union of cells, $\alpha$ is an embedded edge path bounding $E$ and $E$ union $\alpha$ is a closed subspace of $[0,\infty)\times [0,1]$ homeomorphic to a closed ball or a closed half space in $[0,\infty)\times [0,1]$. When $\alpha$ is finite, we say the disk pair is finite, otherwise we say it is unbounded. Note that if $(E,\beta)$ is a disk pair, then $\beta$ is collared in $E$.
\end{definition}

When there is enough information about the $Z_i$, we intend to apply the theorem to alter the homotopy $M$ to a proper homotopy in $Y$. 

\begin{theorem} \label{excise} (Theorem 6.1, \cite{MS19}) %{\bf(excise)}
Suppose $M:[0,\infty)\times [0,1]\to  X$ is a proper simplicial homotopy $rel\{\ast\}$ of proper edge path rays  $r$ and $s$ into a connected 
locally finite simplicial 2-complex $X$, where $r$ and $s$ have image in a subcomplex $Y$ of $X$. Say $\mathcal Z=\{Z_i\}_{i=1}^\infty$ is a collection of connected subcomplexes of $Y$ such that only finitely many $Z_i$ intersect any compact subset of $X$. 
Assume that each vertex of $X-Y$ is separated from $Y$ by exactly one $Z_i$. 

Then there is an index set $J$ and for each $j\in J$, there is a disk pair $(E_j,\alpha_j)$ in  $[0,\infty)\times [0,1]$ where the $E_j$ are disjoint, $M$ maps $\alpha_j$  to $Z_{i(j)}$ (for some $i(j)\in \{1,2,\ldots\}$) and $M([0,\infty)\times [0,1]-\cup_{j\in J} E_j)\subset Y$. 
\end{theorem} 

\begin{remark}\label{Tfacts} (Remark 6.2, \cite{MS19}) %{\bf(Tfacts)}
Assume the hypotheses of Theorem \ref{excise}. If $e$ is an edge of $[0,\infty)\times [0,1]$ then there are at most two $j\in J$   such that $e$ is an edge of $\alpha_j$. Hence if $K$ is a finite subcomplex of $[0,\infty)\times [0,1]$, there are only finitely many $j\in J$ such that $\alpha_j$ has an edge in $K$. 

If $Z_i\in \mathcal Z$ is a finite subcomplex of $X$  and $M$ maps $\alpha_j$ to $Z_i$ then since $M$ is proper, $\alpha_j$ is a circle (and not a line), so that $E_j$ is bounded in $[0,\infty)\times [0,1]$. Also since $M$ is proper,  $M^{-1}(Z_i)$ is compact, and so $M$ maps only finitely many $\alpha_j$ to $Z_i$.  
\end{remark}

The Simplicial Approximation Theorem (see [Theorem 3.4.9, \cite{Span66}]) applies to finite simplicial complexes. The next result is an elementary extension of that result. 

\begin{lemma} \label{simpA} (Lemma 6.5, \cite{MS19}) %{\bf (simpA)}
Suppose $M:[0,\infty)\times [0,1]\to X$ is a proper map to a simplicial complex $X$ where $M_0$ and $M_1$ are (proper) edge paths and $M(0,t)=M(0,0)$ for all $t\in [0,1]$. Then there is a proper simplicial approximation $M'$  of $M$ that agrees with $M$ on $([0,\infty)\times \{0,1\} )\cup (\{0\}\times [0,1])$. 
\end{lemma} 

 \begin{proof} {\bf (Of Theorem \ref{splitF})} 
 Suppose the  groups $A$, $B$ and $G=A\ast_CB$ are finitely presented, $C$ is finite, $G$ has semistable fundamental group at $\infty$ and $A$ does not. Choose finite presentations $\mathcal P_1=\langle \mathcal A\cup \mathcal C:\mathcal R_1\rangle$ and $\mathcal P_2=\langle \mathcal B\cup \mathcal C':\mathcal R_2\rangle$ for $A$ and $B$ respectively, where $\mathcal C=\{c_1,\ldots, c_n\}$ and $\mathcal C'=\{c_1',\ldots, c_n'\}$ generate $C$ in $A$ and $B$ respectively. Also assume that in $G$, $c_i=c_i'$ for all $i\in\{1,\ldots, n\}$ and $\mathcal R_1$ contains relations $Q_1$ such that $\mathcal P_3=\langle \mathcal C:Q_1\rangle$ finitely presents  $C$. 

Let $S=\mathcal A\cup \mathcal B\cup \mathcal C$ and $R=\mathcal R_1\cup \mathcal R_2$. Then a presentation for $G$ is $\mathcal P=\langle S:R\rangle$ where each $c_i'$ in a relation of $R_2$ is replaced by $c_i$. Let $X$ be the standard 2-complex for $\mathcal P$ (one vertex $\ast$, a labeled loop at $\ast$ for each element of $S$, and a 2-cell attached at $\ast$ according to  each relation $r\in R$).

Let $\tilde X$ be the universal cover of $X$. The space $\tilde X$ is the Cayley 2-complex for $(G,\mathcal P)$ with 1-skeleton the (labeled) Cayley graph for $(G, S)$. Each vertex of $\tilde X$ is mapped by the covering map $p:\tilde X\to X$ to $\ast$, each edge of $\tilde X$ is mapped to a loop with the same label and each 2-cell of $\tilde X$ is mapped to a 2-cell with the same boundary labels. 

Let  $\ast$ be the identity vertex of $\tilde X$, let $Y_\ast$ be the maximal (full) subcomplex of $\tilde X$ containing $\ast$ and such that each edge of $Y_\ast$ is labeled by a letter in $\mathcal A\cup\mathcal C$. The space  $Y_\ast$ is the (simply connected) Cayley 2-complex for $(A, \mathcal P_1)$.  

If $v\in A$, let $Z_v$ be the maximal full subspace of $Y_\ast$ containing $v$ such that each edge of $Z_v$ is labeled by an element of $C$, then $Z_v\subset Y_\ast$, and $Z_\ast$ is a copy of the Cayley 2-complex of $(C,\mathcal P_3)$. Since $C$ is a finite group, $Z_v$ is a finite (simply connected) complex.  %The space $\tilde X$ is a ``tree of spaces" modeled on $\mathcal T$, the Bass-Serre tree for the splitting $G=A\ast_CB$. The vertices of $\mathcal T$ are the cosets $gA$ and $gB$ for $g\in G$
%The edges of $\mathcal T$ are labeled by the cosets $gC$ for $g\in G$ where $gA$ and $gB$ are the vertices of the edge labeled $gC$. Let $Z=Z_\ast$ and $Y=Y_\ast$ where $\ast$ represents the identity vertex of $\tilde X$ (so $Z_g=gZ$ and $Y_g=gY$). For each $g\in G$, the (open) edge $gZ$ of $\mathcal T$ separates $\mathcal T$ into two subtrees. If $v$ and $w$ are vertices of $\tilde X$, belonging to an $A$ or $B$ coset of $G$ on opposite sides of the edge labeled $gZ$ of $\mathcal T$, then any edge path in $\tilde X$ from $v$ to $w$ must pass through $gZ$. 

We begin with a triangulation of $\tilde X$ that refines our cellular structure. We may assume there are no relations in $R$ of length 1.
Add a vertex to the interior of every edge (so the edge is exchanged for two edges). Now, every  2-cell is bounded by at least 4 edges. 
Triangulate each 2-cell, by selecting a point in the interior of each 2-cell to be a vertex, and defining edges from the new vertex to each of the other vertices of  that cell. 
This defines a triangulation of $\tilde X$. 

There are proper edge path rays $r$ and $s$ in $Y$, such that $r$ and $s$ are based at $\ast$, converge to the same end of $Y$  and are not properly homotopic in $Y$ (since $A$ is not semistable). As $G$ is semistable, the proper rays $r$ and $s$ (which converge to the same end of $\tilde X$) are properly homotopic $rel\{v\}$ in $\tilde X$. Suppose $H:[0,\infty)\times [0,1]$ is a proper  homotopy $rel\{v\}$ of $r$ to $s$ in $\tilde X$.

Lemma \ref{simpA} implies there is a simplicial structure on $[0,\infty)\times [0,1]$  and a proper simplicial homotopy $M:[0,\infty)\times [0,1]\to \tilde X$ of $r$ to $s$.  We  will apply Theorem \ref{excise} to $M$ and then obtain  a proper homotopy of $r$ to $s$ with image in $Y$ (the desired contradiction).

%For each equivalence class $\mathcal C$ of red triangles in $[0,\infty)\times [0,1]$, consider the edge path loop $\alpha_{\mathcal C}$ bounding the open disk $D_{\mathcal C}$. Since $\alpha_{\mathcal C}$ has image in $a_{\mathcal C}Z$ for some $a_{\mathcal C}\in A$, we can change the homotopy $M$ on each $D_{\mathcal C}$ to kill each loop $\alpha_{\mathcal C}$ by a map into $a_{\mathcal C}Z$. Call the resulting homotopy $M_1$. Note that $M_1$ maps $[0,\infty)\times [0,1]$ into $Y$. It remains to show that $M_1$ is proper. 

In the hypothesis of Theorem \ref{excise}, let $\tilde X$ play the role of $X$, let $Y_\ast$ play the role of $Y$. Since $C$ is finite, $C$ has infinite index in $A$ (otherwise $A$ is finite and finite groups are all semistable). Let $\{v_1C, v_2C, \ldots \}$ be the $C$-cosets of $A$.  Let $Z_i=Z_{v_i}$ and $\{Z_i\}_{i=1}^\infty$ play the role of $\mathcal Z$ in Theorem \ref{excise}.  By Theorem \ref{excise} there are open subsets $E_j$ for $j\in J$, each homeomorphic to $\mathbb R^2$ and edge paths $\alpha_j$ bounding $E_j$ such that $(E_j,\alpha_j)$ is a disk pair. Since each $Z_i$ is finite,  Remark \ref{Tfacts} guarantees that each edge path $\alpha_j$ is finite. By Theorem \ref{excise}, $M$ restricted to $[0,\infty)\times [0,1]-\cup_{i=1}^\infty E_j$ has image in $Y$ and for each $j$, $M(\alpha_j)$ has image in say $Z_{i(j)}$. Let $M_1$ be the homotopy that agrees with $M$ on $[0,\infty)\times [0,1]-\cup_{i=1}^\infty E_j$ and on $E_j$ is a simplicial homotopy that kills $\alpha_j$ in $Z_{i(j)}$. To complete the proof of Theorem \ref{splitF} we need to show that $M_1$ is proper. 
 
If $K$ is a finite subcomplex in $Y$ and $M_1^{-1}(K)$ is not a finite complex, then there is an unbounded sequence of vertices in $[0,\infty)\times [0,1]$ that are mapped by $M_1$ into $K$. 
Since $M$ is proper, there is an unbounded sequence of vertices in the union of the disks $E_j$ that $M_1$ maps into $K$. Since each $E_j$ contains only finitely many vertices,  there are infinitely many $E_j$ each of which has a vertex that $M_1$ maps to $K$. But Remark \ref{Tfacts} guarantees that only finitely many disks are mapped by $M_1$ to a given $aZ$ for $a\in A$. This means that infinitely many distinct $aZ$ intersect $K$, a contradiction. 
 \end{proof}

\bibliographystyle{amsalpha}
\bibliography{paper1}{}

\end{document}